\documentclass[11pt]{amsart}

\usepackage{amssymb}

\input xy

\xyoption{all}

\theoremstyle{plain}

\newtheorem{theorem}{Theorem}[subsection]

\newtheorem{corollary}[theorem]{Corollary}

\newtheorem{lemma}[theorem]{Lemma}

\newtheorem{proposition}[theorem]{Proposition}

\newtheorem{fact}[theorem]{Fact}

\newtheorem{claim}[theorem]{Claim}

\newtheorem{question}[theorem]{Question}

\theoremstyle{definition} %actually mainly don't "mute" (in hebrew)

\newtheorem{definition}[theorem]{Definition}

\newtheorem{remark}[theorem]{Remark}

\theoremstyle{remark}

\newtheorem{hypothesis}[theorem]{Hypothesis}

\renewcommand{\phi}{\varphi}

\newcommand{\initial}\lessdot

\def\?{?\vadjust

{\vbox to 0pt{\vskip-7pt\hbox to 1.1\hsize{\hfill\huge ?!}}}}

\newcommand{\be}{\begin{enumerate}}
\newcommand{\ee}{\end{enumerate}}

\usepackage{enumerate}
\usepackage{verbatim}
\renewcommand{\epsilon}{\varepsilon}

 \def\nfork{\setbox0\hbox{$\bigcup$}%
 \setbox1=\hbox to \wd0{\hfil\vrule width 0.7pt depth 2pt height 7.5pt\hfil}%
 \wd1=0cm\relax\box1\box0}
 \def\dnf{\mathop{\nfork}\limits}

\begin{document}
%%%%%%%%%%%%%%%%%%%%%%%%%%%%%%%%
%%%%%%%%%%%%%%%%%%%%%%%%%%%%%%%%
\title{A semi-good frame with amalgamation and tameness in $\lambda^+$}

\author{Adi Jarden}
\email[Adi Jarden]{jardenadi@gmail.com}
\address{Department of Computer Science and Mathematics \\ Ariel University Center of Samaria \\ Ariel, Israel}
\address{Department of Mathematics \\ Bar-Ilan University \\ Ramat-Gan, Israel}

\begin{abstract}
We introduce a connection between tameness and non-forking frames. We assume the existence of a semi-good non-forking $\lambda$-frame, $(\lambda,\lambda^+)$-tameness and amalgamation in $\lambda^+$ and present sufficient additional conditions for the existence of a good non-forking $\lambda^+$-frame. 

Moreover, we improve results of \cite{jrsi3} about independence and dimension. 
\end{abstract}

\maketitle

\section{Introduction}
In \cite{shh}.III Shelah presented an axiomization of a non-forking relation in the context of AECs. But this non-forking relation relates to models of a specific cardinality, $\lambda$, only. Shelah presented a way to extend a non-forking relation to models of cardinality $>\lambda$ and proved that several axioms preserved. 

The extension, uniqueness and symmetry are the problematic axioms. But uniqueness for models of cardinality $>\lambda$ is actually tameness. It is known that we can derive extension from uniqueness in the smaller cardinalities. So our main challenge is to get symmetry.  Some years ago, we conjectured that the symmetry holds too, but it is still an open question.

In \cite{bo complete...}, Boney presented a variant of tameness that is a sufficient condition for symmetry. Later, Sebastien  proved a non-structure theorem assuming the failure of symmetry in a similar context.  It should be checked, if we can apply the non-structure theorem of Sebastien.

Here, we present a new sufficient condition for symmetry: we show that a variant of the continuity property for independence (studied in \cite{jrsi3}) is a sufficient condition and prove it under reasonable hypothesis. 

In \cite{shh}.III and \cite{jrsh875} we derive good $\lambda^+$-frames too. Let us compare the main results of those papers with the main result in the current paper. The advantages of the current paper are:
\begin{enumerate}
\item We do not restrict our selves to the saturated models, \item we do not restrict the relation $\preceq$ to $\preceq^{NF}_{\lambda^+}$ and \item we do not assume that $I(\lambda^{++},K)<2^{\lambda^{++}}$. 
\end{enumerate}

The disadvantages of the current paper are the following hypotheses:
\begin{enumerate}
\item the amalgamation property in $\lambda^+$ and
\item tameness.
\end{enumerate}

\section{Non-forking frames}

In \cite{shh}.III, Shelah introduced the notion of a good (non-forking) $\lambda$-frame. It is an axiomatization of non-forking. In Definition \cite[2.1.3]{jrsh875}, good frames generalized to semi-good frames: the stability hypothesis is weakened. From now we assume:
\begin{hypothesis}\label{hypothesis 1}
$\frak{s}=(K,\preceq,\dnf,S^{bs})$ is a semi-good non-forking frame.
\end{hypothesis}

\begin{remark}\label{local character alomost implies continuity}
By \cite{jrsh940}, without loss of generality, for each $M \in K_\lambda$ $S^{bs}(M)=S^{na}(M)$, namely, the basic types are the non-algebraic types. 
%In this case, the
%continuity axiom is an easy consequence of the local character.
\end{remark}

\subsection{Non-forking with larger models}
We recall \cite[Definition 2.6.1]{jrsh875}, where we extend the non-forking relation to include models of
cardinality greater than $\lambda$.

\begin{definition}\label{preparation for forking for big models}
\label{2.9} $\dnf^{\geq \lambda}$ is the class of quadruples
$(M_0,a,M_1,M_2)$ such that:
\begin{enumerate}
\item $\lambda \leq ||M_i||$ for each $i<3$.\item $M_0 \preceq M_1
\preceq M_2$ and $a \in M_2-M_1$. \item For some model $N_0 \in
K_\lambda$ with $N_0 \preceq M_0$ for each model $N \in
K_\lambda$, $N_0 \bigcup \{a\} \subseteq N \preceq M_1 \Rightarrow
\dnf(N_0,a,N,M_2)$.
\end{enumerate}
\end{definition}

\begin{definition}\label{forking for big models}
Let $M_0,M_1$ be models in $K_{\geq \lambda}$ with $M_0 \preceq
M_1$ and $p \in S(M_1)$. We say that $p$ does not fork over $M_0$,
when for some triple $(M_1,M_2,a) \in p$ we have $\dnf^{\leq
\lambda}(M_0,a,M_1,M_2)$.
\end{definition}

\begin{remark}
We can replace the quantification `for some' ($M_1,M_2,a$) in
Definition \ref{forking for big models} by `for each'.
\end{remark}

\begin{definition}\label{basic for big models}
Let $M \in K_{>\lambda},\ p \in S(M)$. $p$ is said to be
\emph{basic} when there is $N \in K_\lambda$ such that $N \preceq
M$ and $p$ does not fork over $N$. For every $M \in K_{>\lambda},\
S^{bs}_{> \lambda}(M)$ is the set of basic types over $M$.
Sometimes we write $S^{bs}_{\geq \lambda}(M)$, meaning $S^{bs}(M)$
or $S^{bs}_{> \lambda}(M)$ (the unique difference is the
cardinality of $M$).
\end{definition}
 
The following fact  is an immediate consequence of \cite[Theorem 2.6.8]{jrsh875}.

\begin{fact}
If $\frak{s}^+$ satisfies basic stability, uniqueness, extension and symmetry then it is a good non-forking frame. 
\end{fact}

From now on we add the following hypothesis:
\begin{hypothesis}\label{hypothesis 2}
\begin{enumerate}
\item The amalgamation property in $\lambda^+$ holds and
\item $(K,\preceq)$ satisfies $(\lambda,\lambda^+)$-tameness.
\end{enumerate}
\end{hypothesis}

\begin{theorem}\label{the main theorem}\label{symmetry implies good-frame}
Suppose Hypotheses \ref{hypothesis 1} and \ref{hypothesis 2}. If $\frak{s^+}$ satisfies symmetry then it is a good non-forking $\lambda^+$-frame.
\end{theorem}

\begin{proof}
By Propositions \ref{uniqueness in s^+},  \ref{extension in s^+} and \ref{stability in s^+}.

\begin{proposition}\label{uniqueness in s^+}
$\frak{s}^+$ satisfies uniqueness.
\end{proposition}

\begin{proof}
By uniqueness for $\frak{s}$ and tameness.
\end{proof}

\begin{proposition}\label{extension in s^+}
$\frak{s}^+$ satisfies extension.
\end{proposition}

\begin{proof}
%I have to replace each N by M^+ and each N_{\lambda^+} by N^+.
It is sufficient to prove that if $M$ is a model of cardinality $\lambda$, $N$ is a model of cardinality $\lambda^+$, $M \preceq N$ and $p \in S^{bs}(M)$ then there is a non-forking extension of $p$ to a type over $N$. Take a filtration $\langle M_\alpha:\alpha<\lambda^+ \rangle$ of $N$ with $M_0=M$. Let $N_0$ be a model of cardinality $\lambda$ such that $M_0 \preceq N_0$ and for some $a \in N_0-M_0$ $tp(a,M_0,N_0)=p$. We choose by induction on $\alpha<\lambda^+$ a model $N_\alpha$ and an embedding $f_\alpha:M_\alpha \to N_\alpha$ such that: 
\begin{enumerate} 
\item $f_0$ is the identity from $M$ to $N_0$,
\item if $\alpha=\beta+1$ then $f_\beta \subseteq f_\alpha$ and $tp(a,f_\alpha[M_\alpha],N_\alpha)$ does not fork over $f_\beta[M_\beta]$ (it is possible by the extension property in $\frak{s}$) and \item if $\alpha$ is a limit ordinal then $N_\alpha=\bigcup_{\beta<\alpha}N_\alpha$ and $f_\alpha=\bigcup_{\beta<\alpha}f_\beta$.  
\end{enumerate}
Note that if $\alpha$ is limit then by continuity in $\frak{s}$, $tp(a,f_\alpha[M_\alpha],N_\alpha)$ does not fork over $f_\beta[M_\beta]$ for each $\beta<\alpha$.

Define $N_{\lambda^+}=:\bigcup_{\alpha<\lambda^+}N_\alpha$ and $f=:\bigcup_{\alpha<\lambda^+}f_\alpha$. 
Since $\dnf$ is closed under isomorphisms, it is sufficient to prove that $tp(a,f[N],N_{\lambda^+})$ does not fork over $M$. Let $M'$ be a model of cardinality $\lambda$ with $M \preceq M' \preceq f[N]$. For some $\alpha<\lambda^+$, we have $M' \subseteq f[M_\alpha]$. But $tp(a,f[M_\alpha],f[N])$ does not fork over $M$. Now use monotonicity of non-forking.
\end{proof}

\begin{proposition}\label{stability in s^+}
$\frak{s}^+$ satisfies basic stability.
\end{proposition}

\begin{proof}
By basic stability in $\frak{s}$ and tameness.
%complete...
\end{proof}

This completes the proof of Theorem \ref{the main theorem}
\end{proof}
%of the main theorem

\section{Symmetry}

Recall (from \cite{jrsi3}):
\begin{definition}\label{independence between elements}
The sequence $\langle a,b \rangle$ is independent in $(M,N)$ means that $\{a,b\} \subseteq N-M$ and for some $M_1,M_2$, we have $M \preceq M_1 \preceq M_2$, $N \preceq M_2$, $a \in M_1$, $tp(a,M,M_1)$ is basic, and the type $tp(b,M_1,M_2)$ does not fork over $M$.
\end{definition}

Using the independence teminology, we can reformulate symmetry as follows: for every $M,N,a,b$ the sequence $\langle a,b \rangle$ is independent in $(M,N)$ if and only if the sequence $\langle b,a \rangle$ is independent in $(M,N)$.

In \cite{jrsi3}, independence is defined for sequences of infinite length too, but since it is not used in the current paper, the reader may ignore the following definition and replace $\beta^*$ by $2$ in Definition \ref{definition of the continuity of serial independence}. Anyway, for future applications, we study the more general case. 
\begin{definition} \label{1.6} \label{definition of independence}
\mbox{}

(a) $\langle M_\alpha,a_\alpha:\alpha<\alpha^* \rangle ^\frown
\langle M_{\alpha^*} \rangle$ is said to be \emph{independent}
over $M$ when:
\begin{enumerate}
\item $\langle M_\alpha:\alpha \leq \alpha^* \rangle$ is an
increasing continuous sequence of models in $K_\lambda$. \item $M
\preceq M_0$. \item For every $\alpha<\alpha^*$, $a_\alpha \in
M_{\alpha+1}-M_\alpha$ and the type
$tp(a_\alpha,M_\alpha,M_{\alpha+1})$ does not fork over $M$.
\end{enumerate}

(b) $\langle a_\alpha:\alpha<\alpha^* \rangle$ is said to be
\emph{independent} in $(M,M_0,N)$ when $M \preceq M_0 \preceq N$,
$\{a_\alpha:\alpha<\alpha^*\} \subseteq N-M$ and for some
increasing continuous sequence $\langle M_\alpha:0<\alpha \leq
\alpha^* \rangle$ and a model $N^+$ the sequence $\langle
M_\alpha,a_\alpha:\alpha<\alpha^* \rangle ^\frown \langle
M_{\alpha^*} \rangle$ is independent over $M$, $N \preceq N^+$ and
$M_{\alpha^*} \preceq N^+$.
\end{definition}

\begin{definition}\label{definition of the continuity of serial independence}
\emph{The $\lambda^+$-continuity of serial independence property} is the following property:
Let $\beta^*<\lambda^+$, $M \in K_{\lambda^+}$, $M \preceq N \in K_{\lambda^+}$ and let $\langle M_\alpha:\alpha<\lambda^+ \rangle$ be a filtration of $M$. If $\langle a_\beta:\beta<\beta^* \rangle$ is independent in $(M_\alpha,N)$ for each $\alpha<\lambda^+$ then it is independent in $(M_{\lambda^+},N)$.
\end{definition}

In the following proposition we can weaken the assumption, so that it will refer to sequences of length $\beta^*=2$ only.
\begin{proposition}\label{continuity implies symmetry}
If the $\lambda^+$-continuity of serial independence property holds then symmetry holds. So $\frak{s}^+$  is a good $\lambda^+$-frame.
\end{proposition}

From now on, our goal is to prove the $\lambda^+$-continuity of serial independence property under sufficient conditions.

%%%%%%%%%%%%%%%%%%%%%%%%%%%%%%%%%%
%%                   section                                                    %%%%%%%
%%%%%%%%%%%%%%%%%%%%%%%%%%%%%%%%%%
\section{A Non-Forking Relation on Models}
By Theorem \ref{continuity of serial independence}, the $\lambda^+$-continuity of serial independence property holds if there is a `non-forking' relation $NF$ on models (a relation $NF$ satisfying $\bigotimes_{NF}$, see Definition \ref{definition of non-forking relation on models}).

In \cite{shh}.III, Shelah defined a non-forking relation $NF$ on models, which is based on the non-forking relation $\dnf$. In \cite{jrsh875}, we presented an equivalent definition of $NF$, such that limit models are not mentioned. The new definition is easier to work with and can be applied even when stability in $\lambda$ does not hold. 
  
In Definition \ref{NF}, we list axioms for a relation $NF$ and denote `the relation $NF$ satisfies the list of the axioms' by $\bigotimes_{NF}$. Fact \ref{if the uniqueness triples satisfy the existence property then NF}, we present sufficient conditions for the existence of a relation $NF$ satisfying $\bigotimes_{NF}$ and respecting $\frak{s}$.

In Definitions \ref{definition of widehat{NF}},\ref{definition of preceq^{NF}} we present two relations, that are based on $NF$.

Lemma \ref{we can use NF} and Proposition \ref{widehat{NF} respects independenc} are the key points to prove the $\lambda^+$-continuity of serial independence property.

\begin{definition}\label{definition of non-forking relation on models}
Let $NF \subseteq \ ^4K_\lambda$. \emph{$\bigotimes_{NF}$} means that the following hold:
\begin{enumerate}[(a)]
\item If $NF(M_0,M_1,M_2,M_3)$ then for each $n\in \{1,2\}$
$M_0\leq M_n\leq M_3$ and $M_1 \cap M_2=M_0$. \item Monotonicity:
if $NF(M_0,M_1,M_2,M_3)$, $N_0=M_0$ and for each $n<3$ $N_n\leq
M_n\wedge N_0\leq N_n\leq N_3, (\exists N^{*})[M_3\leq N^{*}\wedge
N_3\leq N^{*}]$ then $NF(N_0 \allowbreak ,N_1,N_2,N_3)$. \item
Extension: For every $N_0,N_1,N_2 \in K_\lambda$, if for each
$l\in \{1,2\}$ $N_0 \leq N_l$ and $N_1\bigcap \allowbreak
N_2=N_0$, then for some $N_3 \in K_\lambda$,
$NF(N_0,N_1,N_2,N_3)$. \item Weak Uniqueness: Suppose for $x=a,b$,
$NF(N_0,N_1,N_2,N^{x}_3)$. Then there is a joint embedding of
$N^a,N^b \ over \ N_1 \bigcup N_2$. \item  Symmetry: For every
$N_0,N_1,N_2,N_3 \in K_\lambda$, $NF(N_0,N_1,N_2,N_3)
\Leftrightarrow NF(N_0,N_2,N_1,N_3)$. \item Long transitivity: For
$x=a,b$, let $\langle  M_{x,i}:i\leq \alpha^* \rangle$ an
increasing continuous sequence of models in $K_\lambda$. Suppose
that for each $i<\alpha^*$, $NF(M_{a,i},\allowbreak
M_{a,i+1},M_{b,i},\allowbreak M_{b,i+1})$. Then
$NF(M_{a,0},M_{a,\alpha^{*}},M_{b,0},M_{b,\alpha^{*}})$. \item $NF$ is closed under isomorphisms: if $NF(M_0,M_1,M_2,M_3)$ and
$f:M_3 \to N_3$ is an isomorphism then
$NF(f[M_0],f[M_1],f[M_2],f[M_3])$.
\end{enumerate}
\end{definition}

\begin{definition}\label{NF respects the frame}
Let $NF$ be a relation such that $\bigotimes_{NF}$ holds. The relation $NF$ \emph{respects the frame} $\frak{s}$ means that if $NF(M_0,M_1,M_2,M_3)$, $a \in M_1-M_0$ and $tp(a,M_0,M_1)$ is basic then $tp(a,M_2,M_3)$ does not fork over $M_0$.
\end{definition}

By Theorem \cite[5.5.4]{jrsh875}
(and Definitions \cite[5.2.1,5.2.6]{jrsh875}):
\begin{fact} \label{jrsh875.5.15}
If the class of uniqueness triples satisfies the existence property
then there is a (unique) relation $NF \subseteq \ ^4K_\lambda$ satisfying $\bigotimes_{NF}$ and respecting the frame $\frak{s}$.
\end{fact}

From now on we assume:
\begin{hypothesis}\label{hypothesis for bar{NF}}
\begin{enumerate}
\item $\frak{s}$ is a semi-good non-forking frame, \item there is a non-forking relation $NF$ repecting $\frak{s}$.
\end{enumerate}
\end{hypothesis}

We define a notion for: a model of size $\lambda$ is independent from a model of size $\lambda^+$ over a model of size $\lambda$ in a model of size $\lambda^+$.
\begin{definition}\label{definition of widehat{NF}}\label{5.14} Define a 4-ary relation $\widehat{NF}$ on $K$ by $$\widehat{NF}(N_0,N_1,M_0,\allowbreak M_1)$$ when the following hold:
\begin{enumerate}
\item $N_0,N_1$ are of cardinality $\lambda$, \item $M_0,M_1$ are of cardinality $\lambda^+$,
\item There are filtrations $\langle
N_{0,\alpha}:\alpha<\lambda^+ \rangle,\ \langle
N_{1,\alpha}:\alpha<\lambda^+ \rangle$ of $M_0,M_1$ respectively, such that
$NF(N_{0,\alpha},N_{1,\alpha},N_{0,\alpha+1},N_{1,\alpha+1})$ holds for every $\alpha<\lambda^+$.
\end{enumerate}
\end{definition}

\begin{fact} [basic properties of $\widehat{NF}$]
\label{5.15}\label{the widehat{NF}-properties} \mbox{}
\begin{enumerate}[(a)] \item Disjointness: If
$\widehat{NF}(N_0,N_1,M_0,M_1)$ then $N_1 \bigcap M_0=N_0$. \item
Monotonicity: Suppose $\widehat{NF}(N_0,N_1,M_0,M_1),\ N_0 \preceq
N^{*}_1 \preceq N_1,\ N_1^* \bigcup M_0 \allowbreak \subseteq
M^*_1 \preceq M_1$ and $M_1^* \in K_{\lambda^+}$. Then
$\widehat{NF}(N_0,N^{*}_1,M_0,M^*_1)$. \item Extension: Suppose
$n<2 \Rightarrow N_n \in K_\lambda,\ M_0 \in K_{\lambda^+},\ N_0
\preceq N_1,\ N_0 \preceq M_0,\ N_1 \bigcap M_0=N_0$. Then there
is a model $M_1$ such that $\widehat{NF}(N_0,N_1,\allowbreak
M_0,M_1)$. \item Weak Uniqueness: If $n<2 \Rightarrow
\widehat{NF}(N_0,N_1,M_0,M_{1,n})$, then there are $M,f_0,f_1$
such that $f_n$ is an embedding of $M_{1,n}$ into $M$ over $N_1
\bigcup M_0$. \item Respecting the frame: Suppose
%$\frak{s}$ satisfies the conjugation property,
$\widehat{NF}(N_0,N_1,M_0,M_1)$ and $tp(a,N_0,M_0) \in \allowbreak S^
{bs} \allowbreak (N_0)$. Then $tp(a,N_1,M_1)$ does not fork over
$N_0$.
\end{enumerate}
\end{fact}

Now we define a relation $\preceq^{NF}_{\lambda^+}$ on
$K_{\lambda^+}$, that is based on the relation $\widehat{NF}$:
\begin{definition} \label{6.1}\label{6.4 in april}
Suppose $M_0,M_1 \in K_{\lambda^+}$, $M_0 \preceq M_1$. Then $M_0
\preceq^{NF}_{\lambda^+} M_1$ means that $\widehat{NF}(N_0,N_1,M_0,M_1)$ for some $N_0,N_1 \in
K_\lambda$.
\end{definition}

\begin{fact} \label{6.2}\label{6.5}\label{preceq^{NF}-properties}
$(K_{\lambda^+},\preceq^{NF}_{\lambda^+})$ satisfies the following
properties:
\begin{enumerate}[(a)]
\item Suppose $M_0 \preceq M_1,\ n<2 \Rightarrow M_n \in
K_{\lambda^+}$. For $n<2$, let $\langle
N_{n,\epsilon}:\epsilon<\lambda^+ \rangle$ be a representation of
$M_n$. Then $M_0 \preceq^{NF}_{\lambda^+}M_1$ iff there is a club
$E \subseteq \lambda^+$ such that $(\epsilon<\zeta \wedge
\{\epsilon,\zeta\} \subseteq E) \Rightarrow
NF(N_{0,\epsilon},N_{0,\zeta},N_{1,\epsilon},N_{1,\zeta})$. \item
$\preceq^{NF}_{\lambda^+}$ is a partial order. \item If $M_0
\preceq M_1 \preceq M_2$ and $M_0 \preceq^{NF}_{\lambda^+} M_2$
then $M_0 \preceq^{NF}_{\lambda^+} M_1$. \item
$(K_{\lambda^+},\preceq^{NF}_{\lambda^+})$ satisfies Axiom c of
AEC in $\lambda^+$, i.e.: If $\delta \in \lambda^{+2}$ is a limit
ordinal and $\langle M_\alpha:\alpha<\delta \rangle$ is a
$\preceq^{NF}_{\lambda^+}$-increasing continuous sequence, then
$M_0 \preceq^{NF}_{\lambda^+} \bigcup \{M_\alpha:\alpha<\delta \}$
and obviously it is $\in K_{\lambda^+}$. \item $K_{\lambda^+}$ has
no $\preceq^{NF}_{\lambda^+}$-maximal model. 
%\item LST for
%$\preceq^{NF}_{\lambda^+}$: If $M_0 %\preceq^{NF}_{\lambda^+} M_1,\
%n<2 \Rightarrow (A_n \subseteq M_n \wedge |A_n| %\leq \lambda)$,
%then there are models $N_0,N_1 \in K_\lambda$ such %that:
%$\widehat{NF}(N_0,N_1,M_0,M_1)$ and $n<2 %\Rightarrow A_n\subseteq
%N_n$.
\end{enumerate}
\end{fact}

\begin{remark}\label{remark preceq^{NF}}
Let $M_1,M_2$ be models of cardinality $\lambda^+$ with $M_1 \preceq M_2$. Then $M_1 \preceq^{NF}_{\lambda^+}M_2$ if and only if for every two filtrations $\langle M_{1,\alpha}:\alpha<\lambda^+ \rangle$ and $\langle M_{2,\alpha}:\alpha<\lambda^+ \rangle$ of $M_1$ and $M_2$ respectively, for some club $E$ of $\lambda^+$ for every $\alpha \in E$ we have $\widehat{NF}(M_{1,\alpha},M_{2,\alpha},M_1,M_2)$. 
\end{remark}

\begin{proof}
One direction holds by definition. We prove the hard direction. Suppose $M_1 \preceq^{NF}_{\lambda^+}M_2$. Let $\langle M_{1,\alpha}:\alpha<\lambda^+ \rangle$ and $\langle M_{2,\alpha}:\alpha<\lambda^+ \rangle$ be two filtrations of $M_1$ and $M_2$ respectively. By Fact \ref{preceq^{NF}-properties}(a), for some club $E$ of $\lambda^+$, for every $\epsilon,\zeta \in E$ if $\epsilon<\zeta$ then $NF(M_{1,\epsilon},M_{1,\zeta},M_{2,\epsilon},M_{2,\zeta})$. Let $\alpha \in E$. Then the filtrations $\langle M_{1,\epsilon}:\epsilon \in E-\alpha \rangle$ and $\langle M_{2,\epsilon}:\epsilon \in E-\alpha \rangle$ wittness that $\widehat{NF}(M_{1,\alpha},M_{2,\alpha},M_1,M_2)$. 
\end{proof}

%The following lemma is a version of \cite[Proposition 8.9]{jrsi3}.
\begin{lemma}\label{we can use NF}
 
For every two models $M,M^+$ of cardinality $\lambda^+$: $$M \preceq M^+ \Leftrightarrow M \preceq^{NF}_{\lambda^+}M^+.$$

%there is a model $N$ such that $M \preceq ^{NF}_{\lambda^+}N$ and $M^+ %\preceq N$.
\end{lemma}

\begin{proof}
By Fact \ref{preceq^{NF}-properties}(c), it is sufficient to find a model $N \in K_{\lambda^+}$ such that $
M \preceq ^{NF}_{\lambda^+}N$ and $M^+ \preceq N$.

Without loss of generality, $M \neq M^+$ (otherwise, $N=:M$ does).
Let $A$ be the class of pairs, $(M_1,M^+_1)$ of models of cardinality $\lambda^+$ with $M_1 \preceq M^+_1$. 
Define a strict partial order, $<_A$
on $A$, by: $(M_1,M^+_1) <_A (M_2,M^+_2)$ when the following hold:
\begin{enumerate}
\item $M_1 \preceq^{NF}_{\lambda^+} M_2$, 
\item $M^+_1 \preceq M^+_2$,
\item $M_2 \cap M^+_1 \neq M_1$.
\end{enumerate}

\begin{displaymath}
\xymatrix{M_2 \ar[r]^{id} & M^+_2 \\ 
M_1 \ar[r]^{id} \ar[u]^{\preceq^{NF}_{\lambda^+}} & M^+_1 \ar[u]^{id}  
}
\end{displaymath}

It is sufficient to find a pair $(N,N^+) \in A$ such that $(M,M^+)<_A(N,N^+)$ and $N=N^+$, becuase it yields $M^+ \preceq N$. So by the following claim, it is sufficient to find a pair $(N,N^+) \in A$ such that $(M,M^+)<_A(N,N^+)$ and $(N,N^+)$ is a  $<_A$-maximal pair in $A$. 
\begin{claim}\label{importance of maximal}
Let $(N,N^+) \in A$. If $(N,N^+)$ is $<_A$-maximal then $N=N^+$.
\end{claim}

\begin{proof}
Let $(N,N^+)$ be a pair in $A$ with $N \neq N^+$. We should prove that $(N,N^+)$ is not $<_A$-maximal. By density of basic types (in $\lambda^+$), for some $a \in N^+-N$ the type $tp(a,N,N^+)$ is basic. So there is $N^- \in K_\lambda$ such that $N^- \preceq N$ and $tp(a,N,N^+)$ does not fork over $N^-$. For some $N^-_1 \in K_\lambda$ and some $b \in N^-_1$ we have $tp(b,N^-,N^-_1)=tp(a,N^-,N^+)$.
\begin{displaymath}
\xymatrix{b \in N^-_1 \ar[rr]^{f} && N_1 \ar[r]^{g} & N_1^+ \\ 
N^- \ar[rr]^{id}  \ar[u]^{id} && N \ar[r]^{id} \ar[u]^{\preceq^{NF}_{\lambda^+}}   & N^+ \ni a \ar[u]^{id}
}
\end{displaymath}
By Fact \ref{the widehat{NF}-properties}(c), for some amalgamation $(id \restriction N,f,N_1)$ of $N^-_1$ and $N$ over $N^-$ we have $\widehat{NF}(N^-,N,f[N^-_1],N_1)$. Since the relation $\widehat{NF}$ respects $\frak{s}$, $tp(f(b),N,N_1)$ does not fork over $N^-$. By uniqueness of non-forking (in $\lambda^+$), $tp(a,N,N^+)=tp(f(b),N,N_1)$. Therefore there is an amalgamation $(id \restriction N^+,g,N_1^+)$ of $N^+$ and $N_1$ over $N$ with $g(f(b))=a$. Now we have $$(N,N^+)<_A(g[N_1],N_1^+)$$ (because $N^+ \preceq N_1^+$, $N \preceq^{NF}_{\lambda^+} g[N_1]$ and $a \in g[N_1] \cap N^+-N$). 
\end{proof}
 
\begin{claim}\label{<_A satisfies axiom c}
If $\langle (M_\alpha,M^+_\alpha):\alpha<\delta \rangle$ is a $<_A$-increasing continuous sequence of pairs in $A$ then $(M_\alpha,M^+_\alpha)<_A(\bigcup_{\alpha<\delta}M_\alpha,\bigcup_{\alpha<\delta}M^+_\alpha)$ for each $\alpha<\delta$.
\end{claim}

\begin{proof}
By smoothness, $\bigcup_{\alpha<\delta}M_\alpha \preceq \bigcup_{\alpha<\delta}N_\alpha$. By \cite[Theorem complete...]{jrsh875}, $M_{\lambda^+} \preceq^{NF}_{\lambda^+} \bigcup_{\alpha<\delta}M^\alpha$. By the definition of AEC, $M_{\lambda^{+}+1} \preceq \bigcup_{\alpha<\delta}N_\alpha$.
\end{proof}

Now we can complete the proof of the lemma, using Claim \ref{importance of maximal}. For the sake of a contradiction, assume that there is no $<_A$-maximal pair. We choose by induction on $\alpha<\lambda^{++}$ a pair $(M_\alpha,M^+_\alpha) \in A$ such that for every $\alpha<\lambda^{++}$, $(M_\alpha,M^+_\alpha) <_A (M_{\alpha+1},M^+_{\alpha+1})$ and for every limit $\alpha<\lambda^{++}$, $M_\alpha=\bigcup_{\beta<\alpha}M_\beta$ and $M^+_\alpha=\bigcup_{\beta<\alpha}M^+_\beta$ (so by Claim \ref{<_A satisfies axiom c}, $(M_\beta,M^+_\beta)<_A(M_\alpha,M^+_\alpha)$ for each $\beta<\alpha$). Define $M_{\lambda^{++}}=:\bigcup_{\alpha<\lambda^{++}}M_\alpha$. The sequences $\langle M_\alpha:\alpha<\lambda^{++} \rangle$ and $\langle M^+_\alpha \cap M_{\lambda^{++}}:\alpha<\lambda^{++} \rangle$ are filtrations of $M_{\lambda^{++}}$. So for some $\alpha<\lambda^{++}$ (actually, for a club of $\alpha$'s) we have $M_\alpha=M^+_\alpha \cap M_{\lambda^{++}}$. So $$ M_\alpha \subseteq M^+_\alpha \cap M_{\alpha+1} \subseteq M^+_\alpha \cap M_{\lambda^{++}}=M_\alpha.$$
Therefore $M^+_\alpha \cap M_{\alpha+1}=M_\alpha$, which is impossible, because $(M_\alpha,M^+_\alpha) <_A (M_{\alpha+1},M^+_{\alpha+1})$. This contradiction shows that there is a $<_A$-extension $(N,N^+)$ of $(M,M^+)$, such that $(N,N^+)$ is $<_A$-maximal. So by Claim \ref{importance of maximal}, $N=N^+$. So $M \preceq^{NF}_{\lambda^+}$ and $M^+ \preceq N^+=N$. Lemma \ref{we can use NF} is proved.
\end{proof}

%By the proof of \cite[Proposition 3.5?]{jrsh875}:
\begin{proposition} \label{a rectangle of models with NF}
Let $\alpha^*,\beta^*$ be two ordinals $\leq \lambda^+$ and let $\langle M_{a,\alpha}:\alpha<\alpha^* \rangle$ and $\langle M_{b,\alpha}:\alpha<\beta^* \rangle$ be two increasing continuous sequence of models of cardinality $\lambda$ such that $M_{a,0}=M_{b,0}$. Then there is a `rectangle of models' $\{M_{\alpha,\beta}:\alpha<\alpha^*,\beta<\beta^* \}$ and a sequence $\{f_\beta:\beta<\beta^* \}$ such that for every $\alpha<\alpha^*$ and $\beta<\beta^*$ the following hold:
\begin{enumerate}
\item $M_{\alpha,\beta} \in K_\lambda$, \item if $\alpha+1<\alpha^*$ then $M_{\alpha,\beta} \preceq M_{\alpha+1,\beta}$, \item if $\beta+1<\beta^*$ then $M_{\alpha,\beta} \preceq M_{\alpha,\beta+1}$, \item if $\alpha$ is a limit ordinal then $M_{\alpha,\beta}=\bigcup_{\alpha'<\alpha}M_{\alpha',\beta}$, \item if $\beta$ is a limit ordinal then $M_{\alpha,\beta}=\bigcup_{\beta'<\beta}M_{\alpha,\beta'}$,  
\item $f_\beta$ is an isomorphism of $M_{b,\beta}$ onto $M_{0,\beta}$, \item if $\beta=0$ then $M_{\alpha,\beta}=M_{a,\alpha}$ and $f_\beta$ is the identity on $M_{a,0}=M_{b,0}$, \item $f_{\beta'} \subseteq f_\beta$ for every $\beta'<\beta$, \item if $\beta$ is a limit ordinal then $f_\beta=\bigcup_{\beta'<\beta}f_{\beta'}$, \item $NF(M_{\alpha,\beta},M_{\alpha,\beta+1},M_{\alpha+1,\beta},M_{\alpha+1,\beta+1})$, unless $\alpha+1=\alpha^*$ or $\beta+1=\beta^*$.
\end{enumerate}  
\end{proposition}

\begin{displaymath}
\xymatrix{M^a_{\alpha+1} \ar[r]^{=} & M_{\alpha+1,0} 
\ar[rrr]^{id} &&& M_{\alpha+1,\beta} \ar[r]^{id} & M_{\alpha+1,\beta+1} \\
M^a_{\alpha} \ar[r]^{=} & M_{\alpha,0} \ar[u]^{id}
\ar[rrr]^{id} &&& M_{\alpha,\beta} \ar[r]^{id} \ar[u]^{id} & M_{\alpha,\beta+1} \ar[u]^{id} \\ \\
M^a_{0} \ar[r]^{=} &  M_{0,0} \ar[uu]^{id}
\ar[rrr]^{id} &&& M_{0,\beta} \ar[r]^{id} \ar[uu]^{id} & M_{0,\beta+1} \ar[uu]^{id} \\
& M^b_0 \ar[u]^{=} &&& M^b_\beta \ar[u]^{f_\beta} & M^b_{\beta+1} \ar[u]^{f_{\beta+1}}
}
\end{displaymath}

\begin{proof}
We $\{M_{\alpha,\beta}:\alpha<\alpha^*\}$ and $f_\beta$ by induction on $\beta$. 

\emph{Case a:} $\beta=0$. In this case, see Clause (7).

\emph{Case b:} $\beta$ is a limit ordinal. In this case, by Clause (5), we must choose $M_{\alpha,\beta}=\bigcup_{\beta'<\beta}M_{\alpha,\beta'}$ and by Clause (9), we must choose $f_\beta=\bigcup_{\beta'<\beta}f_{\beta'}$. Fix $\alpha<\alpha^*$. By Clauses (3) and (5) of the induction hypothesis, the sequences $\langle M_{\alpha,\beta'}:\beta'<\beta \rangle$ and $\langle M_{\alpha+1,\beta'}:\beta'<\beta \rangle$ are increasing and continuous. So by smoothness $M_{\alpha,\beta} \preceq M_{\alpha+1,\beta}$. Similarly, $f_\beta[M^b_\beta] \preceq M_{0,\beta}$. It remains to show that if $\alpha$ is limit then $M_{\alpha,\beta}=\bigcup_{\alpha'<\alpha}M_{\alpha',\beta}$. On one hand, if $x \in M_{\alpha',\beta}$ for some $\alpha'<\alpha$ then by the induction hypothesis, $x \in M_{\alpha',\beta'}$ for some $\beta'<\beta$. But $M_{\alpha',\beta'} \subseteq M_{\alpha,\beta'} \subseteq M_{\alpha,\beta}$. On the other hand, if $x \in M_{\alpha,\beta}$ then $x \in M_{\alpha,\beta'}$ for some $\beta'<\beta$. Therefore $x \in M_{\alpha',\beta'}$ for some $\alpha'<\alpha$. But $M_{\alpha',\beta'} \subseteq M_{\alpha',\beta}$ compelte...

\emph{Case c:} $\beta=\gamma+1$. complete...
\end{proof}

\begin{proposition} \label{widehat{NF} with independence}
If the sequence $\langle N_\alpha,a_\alpha:\alpha<\alpha^* \rangle ^\frown \langle N_{\alpha^*} \rangle$ is independent over $N_0$ and $N_0 \preceq N_0^* \in K_{\lambda^+}$ then for some $N_1^* \in K_{\lambda^+}$ and some embedding $f:N_0^* \to N_1^*$, the sequence $\langle a_\alpha:\alpha<\alpha^* \rangle$ is independent in $(f[N_0^*],N_1^*)$ and  $\widehat{NF}(N_0,N_{\alpha^*},f[N_0^*],N_1^*)$. 
\end{proposition}

\begin{proof}
For every $\beta<\lambda^+$ and every $\alpha \leq \alpha^*$, we choose $N_{\alpha,\beta} \in K_\lambda$ and an embedding $f_{\alpha,\beta}:N_{\alpha,\beta} \to N_{\alpha,\beta+1}$ such that the following hold:
\begin{enumerate}
\item if $\alpha_1<\alpha_2 \leq \alpha^*$ and $\beta_1 \leq \beta_2$ then complete... 
\end{enumerate}

\end{proof}

\begin{proposition}\label{widehat{NF} implies independence}\label{widehat{NF} respects independence}
Let $\alpha^*<\lambda^+$. If $\widehat{NF}(M_0^-,M_1^-,M_0,M_1)$ and $\langle a_\alpha:\alpha<\alpha^* \rangle$ is independent in $(M_0^-,M_1^-)$ then it is independent in $(M_0,M_1)$.
\end{proposition}

\begin{proof}
The idea is to find an amalgam $M_2$ of $M_1^-$ and $M_0$ over $M^-_0$ such that $\widehat{NF}$ holds and the sequence is independent in $(M_0,M_2)$. It is sufficient since the relation $\widehat{NF}$ satisfies weak uniqueness (by complete...).

We elavorate. 
By the definition of independence, there is an increasing continuous sequence $\langle N_\alpha:\alpha \leq \alpha^* \rangle$ of models of cardinality $\lambda$ such that $N_0=M_0^-$, $M_1^- \preceq N_{\alpha^*}$ and the sequence $\langle M_\alpha,a_\alpha:\alpha<\alpha^* \rangle ^\frown \langle M_{\alpha^*} \rangle$ is independent over $M_0^-$. By Proposition \ref{widehat{NF} with independence}, for some $N_1^* \in K_{\lambda^+}$ the sequence $\langle a_\alpha:\alpha<\alpha^* \rangle$ is independent in $(M_0,N_1^*)$ and  $\widehat{NF}(N_0,N_{\alpha^*},M_0,N_1^*)$. 

But $\widehat{NF}(M_0^-,M_1^-,M_0,M_1)$. Therefore by the weak uniqueness of $\widehat{NF}$ (Fact \ref{complete...}) complete...

complete...
\end{proof}
 
\begin{theorem}\label{continuity of serial independence}\label{if the class of uniqueness triples satisfies the existence property then continuity}
The $\lambda^+$-continuity of serial independence property holds.
\end{theorem}

\begin{proof}
Let $\beta^*<\lambda^+$, $M_1 \in K_{\lambda^+}$, $M_1 \preceq M_2 \in K_{\lambda^+}$ and let $\langle M_{1,\alpha}:\alpha<\lambda^+ \rangle$ be a filtration of $M_1$. Suppose $\langle a_\beta:\beta<\beta^* \rangle$ is independent in $(M_{1,\alpha},M_2)$ for each $\alpha<\lambda^+$. We have to prove that $\langle a_\beta:\beta<\beta^* \rangle$ is independent in $(M_1,M_2)$.

By Lemma \ref{we can use NF}, $M_1 \preceq^{NF}_{\lambda^+}M_2$. Let $\langle M_{2,\alpha}:\alpha<\lambda^+ \rangle$ be a filtration of $M_2$. By Remark \ref{remark preceq^{NF}}, there is a club $E$ of $\lambda^+$ such that for every $\alpha \in E$, $\widehat{NF}(M_{1,\alpha},M_{2,\alpha},M_1,M_2)$.   Define $J=:\{a_\beta:\beta<\beta^*\}$. $J \subseteq M_2$. Since $|J|<\lambda^+$, for some $\alpha \in E$ we have $J \subseteq M_{2,\alpha}$. But 
%by Fact \ref{the widehat{NF}-properties}, 
$\widehat{NF}(M_{1,\alpha},N_{2,\alpha},M_1,M_2)$. So by Proposition \ref{widehat{NF} implies independence}, $\langle a_\beta:\beta<\beta^* \rangle$ is independent in $(M_1,M_2)$.
\end{proof}

\begin{proposition}\label{continuity implies symmetry}
Suppose:
\begin{enumerate}
\item $\frak{s}=(K,\preceq,S^{bs},\dnf)$ is a semi-good $\lambda$-frame \item $(K,\preceq)$ satisfies the amalgamation property in $[\lambda,\mu]$ \item $(K,\preceq)$ satisfies $(\lambda,\mu)$-tameness \item $\frak{s}$ satisfies the continuity of serial independence property.
\end{enumerate}
Then $\frak{s}^{\mu}$ satisfies symmetry.
\end{proposition}

\begin{proof}
%If $\lambda^+<\mu$ then the theorem might be wrong. complete... Suppose $\mu=\lambda^+$.

Suppose $N_0,N_1,N_2$ are models of cardinality $\mu$, $N_0 \preceq N_1 \preceq N_2$, the type $tp(a,N_0,N_1)$ is basic and $tp(b,N_1,N_2)$ does not fork over $N_0$. Let $\langle M_{0,\alpha}:\alpha<\lambda^+ \rangle$, $\langle M_{1,\alpha}:\alpha<\lambda^+ \rangle$, $\langle M_{2,\alpha}:\alpha<\lambda^+ \rangle$ be filtrations of $N_0,N_1,N_2$ respectively such that $M_{0,\alpha} \preceq M_{1,\alpha} \preceq M_{2,\alpha}$ for each $\alpha<\lambda^+$. Without loss of generality, the types $tp(a,N_0,N_1)$, $tp(b,N_0,N_2)$ do not fork over $M_{0,0}$, $a \in M_{1,0}$ and $b \in M_{2,0}$.

The sequence $\langle a,b \rangle$ is independent in $(M_{0,0},M_{0,\alpha},M_{2,\alpha})$ for each $\alpha<\lambda^+$. 
So by symmetry in $\frak{s}$, the sequence $\langle b,a \rangle$ is independent in $(M_{0,0},M_{0,\alpha},M_{2,\alpha})$ for each $\alpha<\lambda^+$. By Theorem \ref{continuity of serial independence} (the $\lambda^+$-continuity of serial independence property), $\langle b,a \rangle$ is independent in $(N_0,N_2)$.
\end{proof}

\begin{corollary}\label{corollary continuity implies good-frame}
Suppose:
\begin{enumerate}
\item $\frak{s}=(K,\preceq,S^{bs},\dnf)$ is a semi-good $\lambda$-frame \item $(K,\preceq)$ satisfies the amalgamation property in $[\lambda,\mu]$ 
\item $(K,\preceq)$ satisfies $(\lambda,\lambda^+)$-tameness 
\item $\frak{s}$ satisfies the continuity of serial independence property.
\end{enumerate}
Then $\frak{s}^{+}$ is a good $\lambda^+$-frame.
\end{corollary}

\begin{proof}
By Proposition \ref{continuity implies symmetry} and Theorem \ref{symmetry implies good-frame}.
\end{proof}

\begin{proof}
By \ref{continuity implies symmetry}, symmetry holds. complete...
\end{proof}

In complete... we generalize the following corollary, eliminating the tameness assumption:
\begin{corollary}
Suppose:
\begin{enumerate}
\item $\frak{s}=(K,\preceq,S^{bs},\dnf)$ is a semi-good $\lambda$-frame \item $(K,\preceq)$ satisfies the amalgamation property in $[\lambda,\mu]$ 
\item $(K,\preceq)$ satisfies $(\lambda,\mu)$-tameness 
\item the class of uniqueness triples satisfies the existence property.
\end{enumerate}
Then $\frak{s}^{\mu}$ is a good $\lambda$-frame.
\end{corollary}

%the following proof holds when we assume tameness, but I delete this assumption in %the corollary. So I write a new proof.
\begin{proof}
By Theorem \ref{if the class of uniqueness triples satisfies the existence property then continuity} and  Corollary \ref{corollary continuity implies good-frame}
\end{proof}

\section{Combining Uniqueness Triples with the Amalgamation Property in $\lambda^+$}

\begin{theorem}
If 
\begin{enumerate}
\item the class of uniqueness triples satisfies the existence property,
\item the amalgamation property in $\lambda^+$ holds.
\end{enumerate}
The $\frak{s}^+$ is a good non-forking $\lambda^+$-frame.
\end{theorem}

We give two proofs. The first one is similar to the proof  Corollary \ref{complete...}, using a replacement of $(\lambda,\lambda^+)$-tameness. 
\begin{proof}
By complete... it is sufficient to prove uniqueness.
\end{proof}

Now we give a second proof, using results from \cite{jrsh875}:
\begin{proof}
By \cite[complete...]{jrsh875}, it is sufficient to prove that the relation $\preceq^{NF}_{\lambda^+}$ satisfies smoothness. But by Lemma \ref{we can use NF}, the relation $\preceq^{NF}_{\lambda^+}$ is equivalent to $\preceq \restriction K_{\lambda^+}$, so by the definition of AEC, it satisfies smoothness.
\end{proof}

\section{Solving the Goodness$^+$ Question}
Some years ago Shelah asked whether for every successful good $\lambda$-frame, $$M_1 \preceq M_2 \Leftrightarrow M_1 \preceq^{NF}_{\lambda^+} M_2$$
holds for every two models $M_1,M_2$ which are saturated in $\lambda^+$ over $\lambda$.

In \cite{shh}.last,complete... he proved it under the assumption that there are no many models of cardinality complete..
Here we solve this open question:
\begin{corollary}
complete...
\end{corollary}

\begin{proof}
Since we restrict ourselves to the saturated models in $\lambda^+$ over $\lambda$, we have categoricity in $\lambda^+$. But we assume $I(\lambda^{++},K)<2^{\lambda^{++}}$. Hence, by Fact \ref{complete...}, we have amalgamation in $\lambda^+$. So by Lemma \ref{we can use NF}, the relations $\preceq \restriction K_{\lambda^+}$ and $\preceq^{NF}_{\lambda^+}$ are equivalent.
\end{proof}

\section{Dimenstion}
In \cite{jrsi3} we proved that if the continuity property holds then the dimension is well-behaved. First note that the continuity property in \cite{jrsi3} relates to independence of sets, while the continuity property here relates to serial independence. In this sense, we study here a stronger continuity property.

While in \cite{jrsi3}, we study $\frak{s}$ (relating to models of cardinality $\lambda$), here we study $\frak{s}^+$. While in \cite{jrsi3}, we assume that the class of uniqueness triples (in $\lambda$) satisfies the existence property, we do not assume the same for $\lambda^+$ (but assume the same assumption). This cause a difficulty.

But by the $\lambda^+$-continuity of serial independence property, we can prove that the dimension is semi-well-behaved in the following sense:
\begin{theorem}
Suppose: complete..., and the class of uniqueness triples in $\lambda$ satisfies the existence property. Let $M,N$ be models of cardinality $\lambda^+$ with $M \preceq N$ and let $J_1,J_2$ be two maximal independent sets in $(M,N)$. Then one of the following hold:
\begin{enumerate}
\item $J_1,J_2$ are finite sets, \item $|J_1|+|J_2|=\lambda^+$ or \item $|J_1|=|J_2|$.
\end{enumerate}
\end{theorem} 

\begin{proof}
complete...
\end{proof}

\end{document}